\pdfoutput=1
\RequirePackage{ifpdf}
\ifpdf 
\documentclass[pdftex]{sigma}
\else
\documentclass{sigma}
\fi

\newtheorem{Theorem}{Theorem}[section]
\newtheorem{Corollary}[Theorem]{Corollary}
\newtheorem{Lemma}[Theorem]{Lemma}

\newtheorem{Conjecture}[Theorem]{Conjecture}
\newtheorem{Question}[Theorem]{Question}
 { \theoremstyle{definition}

 }

\numberwithin{equation}{section}

\def\BN{\mathbb N}
\def\BZ{\mathbb Z}
\def\BQ{\mathbb Q}

\def\BC{\mathbb C}

\def\BP{\mathbb P}

\def\calF{\mathcal F}
\def\calL{\mathcal L}
\def\calO{\mathcal O}

\def\calK{\mathcal K}

\def\calP{\mathcal P}

\def\calM{\mathcal M}
\def\s{\sigma}
\def\la{\langle}
\def\ra{\rangle}

\def\SL{{\rm SL}}
\def\SU{{\rm SU}}
\def\pt{\partial}

\def\a{\alpha}
\def\b{\beta}

\def\ve{\varepsilon}

\def\coeff{\operatorname{coeff}}
\def\Li{\operatorname{Li}}

\def\GV{\operatorname{GV}}
\def\GW{\operatorname{GW}}
\def\Exp{\operatorname{Exp}}
\def\ch{\operatorname{ch}}
\def\End{\operatorname{End}}

\begin{document}
\allowdisplaybreaks

\newcommand{\arXivNumber}{2101.07490}

\renewcommand{\PaperNumber}{021}

\FirstPageHeading

\ShortArticleName{On the Quantum K-Theory of the Quintic}

\ArticleName{On the Quantum K-Theory of the Quintic}

\Author{Stavros GAROUFALIDIS~$^{\rm a}$ and Emanuel SCHEIDEGGER~$^{\rm b}$}

\AuthorNameForHeading{S.~Garoufalidis and E.~Scheidegger}

\Address{$^{\rm a)}$~International Center for Mathematics, Department of Mathematics,\\
\hphantom{$^{\rm a)}$}~Southern University of Science and Technology, Shenzhen, China}
\EmailD{\href{mailto:stavros@mpim-bonn.mpg.de}{stavros@mpim-bonn.mpg.de}}
\URLaddressD{\url{http://people.mpim-bonn.mpg.de/stavros}}

\Address{$^{\rm b)}$~Beijing International Center for Mathematical Research, Peking University, Beijing, China}
\EmailD{\href{mailto:esche@bicmr.pku.edu.cn}{esche@bicmr.pku.edu.cn}}

\ArticleDates{Received October 21, 2021, in final form March 03, 2022; Published online March 21, 2022}

\Abstract{Quantum K-theory of a smooth projective variety at genus zero is a~collection of integers that can be assembled into a generating series $J(Q,q,t)$ that satisfies a system of linear differential equations with respect to $t$ and $q$-difference equations with respect to~$Q$. With some mild assumptions on the variety, it is known that the full theory can be reconstructed from its small $J$-function $J(Q,q,0)$ which, in the case of Fano manifolds, is a vector-valued $q$-hypergeometric function. On the other hand, for the quintic 3-fold we formulate an explicit conjecture for the small $J$-function and its small linear $q$-difference equation expressed linearly in terms of the Gopakumar--Vafa invariants. Unlike the case of quantum knot invariants, and the case of Fano manifolds, the coefficients of the small linear $q$-difference equations are not Laurent polynomials, but rather analytic functions in two variables determined linearly by the Gopakumar--Vafa invariants of the quintic. Our conjecture for the small $J$-function agrees with a proposal of Jockers--Mayr.}

\Keywords{quantum K-theory; quantum cohomology; quintic; Calabi--Yau manifolds; Gro\-mov--Witten invariants; Gopakumar--Vafa invariants; $q$-difference equations; $q$-Frobenius method; $J$-function; reconstruction; gauged linear $\sigma$ models; 3d-3d correspondence; Chern--Simons theory; $q$-holonomic functions}

\Classification{14N35; 53D45; 39A13; 19E20}

\section{Introduction}\label{sec.intro}

\subsection[Quantum K-theory, the small J-function and its q-difference equation]{Quantum K-theory, the small $\boldsymbol{J}$-function and its $\boldsymbol{q}$-difference equation}\label{sub.qk}

The K-theoretic Gromov--Witten invariants of a compact K\"ahler manifold
$X$ (often omitted from the notation) is a collection of integers
(see~\cite[p.~6]{Givental-Tonita})
\begin{gather}\label{EL}
\big\langle E_1 L^{k_1}, \dots, E_n L^{k_n} \big\rangle_{g,n,d}
\end{gather}
defined for vector bundles $E_1,\dots,E_n$ on $X$ and nonnegative integers
$k_1,\dots,k_n$ as the holomorphic Euler characteristic of
$\calO^{{\rm vir}} \otimes
\big( {\otimes}_{i=1}^n {\rm ev}_i^*(E_i) \otimes L_i^{k_1}\big)$ over
the moduli space $\overline{\calM}^{X,d}_{g,n}$ of genus $g$ degree $d$
stable maps to $X$ with $n$ marked points. Here, $L_1,\dots,L_n$ denote the
line (orbi)bundles over $\overline{\calM}^{X,d}_{g,n}$ formed
by the cotangent lines to the curves at the respective marked points.
A definition of these integers was given by Givental and
Lee~\cite{Givental:WDVV, Lee:foundations}. These numerical invariants can be
assembled into a generating series which at genus zero can be used to define
an associative deformation of the product of the K-theory ring $K(X)$ of $X$.

There are several ways to assemble the integers~\eqref{EL} into generating series,
and reconstruction theorems relate these generating series and often determine
one from the other. This is reviewed in Section~\ref{sub.reconstruct}.
Our choice of generating series will be the so-called small $J$-function
\begin{gather}\label{smallJ}
J_X(Q,q,0) = (1-q)\Phi_0 + \sum_{d} \sum_\a
\left\la \frac{\Phi_\a}{1-qL} \right\ra_{0,1,d} \Phi^\a Q^d
\in K(X) \otimes \calK_-(q)[[Q]]
\end{gather}
(with the notation of Section~\ref{sub.notation}), which determines the genus~0 quantum K-theory $X$, i.e., the
integers~\eqref{EL}~\cite[Theorem~1.1, Lemma~3.3]{Iritani:rec} with $g=0$,
as well as the genus 0 permutation-equivariant quantum K-theory
$X$~\cite{Givental:equivariantVIII} (when $K(X)$ is generated by line bundles).

The small $J$-function is a vector-valued function (taking values in the rational
vector space $K(X)$) that obeys a system of linear $q$-difference
equations~\cite{Givental-Lee,Givental-Tonita}, giving rise to
matrices $A_i(Q,q,0) \in K(X) \otimes \calK_+(q)[[Q]]$, for $i=1,\dots,r$
which can also be used to reconstruct the genus~$0$ quantum K-theory of
$X$~\cite[Lemma~3.3]{Iritani:rec}. Concretely, for $X=\BC\BP^N$, the small
$J$-function is given by a~$q$-hypergeometric
formula~\cite{Givental-Lee,Givental-Tonita, Lee:foundations}
\begin{gather}\label{eq.JCPN}
J_{\BC\BP^N}(Q,q,0) = (1-q)\sum_{d=0}^\infty \frac{Q^d}{((1-x)q;q)_d^{N+1}}
\in K\big(\BC\BP^N\big)\otimes \calK_-(q)[[Q]],
\end{gather}
where $(z;q)_d=\prod_{j=0}^{d-1}\big(1-q^jz\big)$ for $d \geq 0$, and
\begin{gather*}
K\big(\BC\BP^N\big) = \BQ[x]/\big(x^{N+1}\big)
\end{gather*}
is the K-theory ring with basis $\big\{1,x,\dots,x^N\big\}$ where $1-x$
is the class of $\calO(1)$.\footnote{The K-theory ring is also written as~\cite[Section~4.1]{Iritani:rec}
$K\big(\BC\BP^N\big) = \BQ\big[P,P^{-1}\big]/\big((1-P)^{N+1}\big)$
as the Grothendieck group of locally free sheaves on projective space,
where $P=\calO_{\BP^N}(-1)$ in which case the small $J$-function takes the
form $J_{\BC\BP^N}(Q,q,0) = (1-q)\sum_{d=0}^\infty \frac{Q^d}{(Pq;q)_d^{N+1}}$.}
The corresponding matrix $A(Q,q,0)$ of the vector-valued $q$-holonomic
function $J(Q,q,0)$ is given by~\cite[Section~4.1]{Iritani:rec}
\begin{gather}\label{eq.ACPN}
A(Q,q,0) = I -
\begin{pmatrix}
 0 & 0 & \ldots & 0 & Q \\
 1 & 0 & \ldots & 0 & 0 \\
 0 & 1 & \ldots & 0 & 0 \\
 \vdots & \vdots & \ddots & \vdots & \vdots \\
 0 & 0 & \ldots & 1 & 0
\end{pmatrix}
\end{gather}
in the above basis of $K(\BC\BP^N)$. It is remarkable that either~\eqref{eq.JCPN}
or~\eqref{eq.ACPN} give the complete determination of all the integers~\eqref{EL}
for $\BC\BP^N$. Observe that the small $J$-function of $\BC\BP^N$ is given by a
vector-valued $q$-hypergeometric formula, which is always $q$-holonomic (as
follows from Zeilberger et al.~\cite{PWZ,WZ,Zeilberger}), and as a result the
entries of $A(Q,q,0)$ (as well as the coefficients of the small quantum product)
are polynomials in $Q$ and $q$. It turns out that the small $J$-function of
Grassmanianns, flag varieties, homogeneous spaces and more generally Fano manifolds
is $q$-hypergeometric as shown by many researchers; see,
e.g.,~\cite{Anderson:finiteness, Taipale, Tonita:quintic} and references therein.
On the other hand, new phenomena are expected for the case of general Calabi--Yau
manifolds, and particularly for the quintic. Our motivation to study the case
of the quintic was two-fold, coming from numerical observations concerning
coincidences of quantum K-theory counts and quantum cohomology counts (given
below), as well as a comparison of the linear $q$-difference equations in quantum
K-theory with those in Chern--Simons theory (such as the $q$-difference equation
of the colored Jones polynomial of a knot~\cite{GL:qholo}).

Our results give a relation between quantum K-theory and quantum cohomology
of the quintic in two different limits, namely $q=1$ (see
Corollary~\ref{cor.numbers}) and $q=0$ (see Corollary~\ref{cor.ctttq}),
and propose a linear expression of the small $J$-function of the quintic
in terms of its Gopakumar--Vafa invariants (see Conjecture~\ref{conj.1}).

\subsection[The small J-function for the quintic]{The small $\boldsymbol{J}$-function for the quintic}\label{sub.Jquintic}

Quantum K-theory was developed by analogy with quantum cohomology (or
Gromov--Witten theory), a theory that deforms the cohomology ring $H(X)$ of $X$
and whose corresponding numerical invariants are rational numbers (known as
Gromov--Witten invariants) or integers in the case of a Calabi--Yau threefold
(known as the Gopakumar--Vafa invariants).
A standard reference is~\cite{Candelas:pair} and the book~\cite{Cox-Katz}.
For the quintic 3-fold $X$, the first six values of the GW and the GV invariants
are given by
\begin{table}[h!]\def\arraystretch{1.3}\centering\small
\begin{tabular}{|c|l|l|l|l|l|l|} \hline
 $d$ & 1 & 2 & 3 & 4 & 5 & 6 \\ \hline
 $\GW_d$ & $\frac{2875}{1}$ & $\frac{4876875}{8}$ & $\frac{8564575000}{27}$ &
 $\frac{15517926796875}{64}$ & $\frac{229305888887648}{1}$
 & $\frac{248249742157695375}{1}$ \\ \hline
 $\GV_d$ & 2875 & 609250 & 317206375 & 242467530000 & 229305888887625
 & 248249742118022000
 \\ \hline
 \end{tabular}
\end{table}

\noindent
with $2875$ being the famous number of rational curves in the quintic.
The two sets of invariants are related by the following multi-covering formula
\begin{gather*}
\GV_n = \sum_{d|n} \frac{\mu(d)}{d^3} \GW_{n/d}, \qquad
\GW_n = \sum_{d|n} \frac{1}{d^3} \GV_{n/d}.
\end{gather*}
In~\cite[Section~6.5]{Tonita:quintic}, Tonita gave an algorithm to
compute the quantum K-theory of the quintic and using it, he found that
\begin{gather*}
\la 1 \ra_{0,1,1}=2875,
\end{gather*}
where 2875 coincides with the famous number of lines in the quintic.
Going further, (see Jockers--Mayr~\cite{Jockers:qK,Jockers:3d}
and equation~\eqref{eq.JJq=0} below) one finds that
 \begin{subequations}\label{012-016}
 \begin{gather}
 \label{012}
\la 1 \ra_{0,1,2} = 620750 = 609250 + 4 \cdot 2875,
\\ \label{013}
\la 1 \ra_{0,1,3} = 317232250 = 317206375 + 9 \cdot 2875,
\\ \label{014}
\la 1 \ra_{0,1,4} = 242470013000 = 242467530000 + 4 \cdot 609250 + 16 \cdot 2875,
\\ \label{015}
\la 1 \ra_{0,1,5} = 229305888959500 = 229305888887625 + 25 \cdot 2875,
\\
 \la 1 \ra_{0,1,6} = 248249743392434250 \nonumber\\
\hphantom{\la 1 \ra_{0,1,6}}{} = 248249742118022000 +4 \cdot 317206375 + 9 \cdot 609250 + 36 \cdot 2875\label{016}
\end{gather}
\end{subequations}
are nearly equal to $\GV$ invariants of the quintic, and more precisely
matched with linear combinations of $\GV$ invariants. Surely this is not
a coincidence and suggests that the $\GV$ invariants can fully reconstruct the
quantum K-theory invariants. In~\cite{Givental-Tonita} this ``coincidence''
is proven in abstractly. Givental and Tonita give a complete solution in
genus-0 to the problem of expressing K-theoretic GW-invariants of
a compact complex algebraic manifold in terms of its cohomological
GW-invariants. One motivation for our work is to give an explicit formula (see
Conjecture~\ref{conj.1} below) of this abstract statement.
To phrase our conjecture, recall that the rational K-theory of the quintic
3-fold $X$ is given by
\begin{gather}\label{eq.KX}
K(X) = \BQ[x]/\big(x^4\big)
\end{gather}
is the K-theory ring with basis $\{\Phi_\a\}$ for $\a=0,1,2,3$ where
$\Phi_\a=x^\a$. Here $1-x$ is the class of~$\calO(1)|_X$. We define
\begin{subequations}\label{eq.a-eq.b}
\begin{gather}\label{eq.a}
5 a(d,r,q) = \frac{dr}{1-q}+\frac{d q}{(1-q)^2},
\\
\label{eq.b}
5 b(d,r,q) = \frac{rd+r^2-d}{1-q}+\frac{d}{(1-q)^2} - \frac{q+q^2}{(1-q)^3}.
\end{gather}
\end{subequations}

\begin{Conjecture} \label{conj.1} The small $J$-function of the quintic is expressed linearly in terms of the GV-invariants by
\begin{gather}\label{Jquintic}
\frac{1}{1-q} J(Q,q,0) = 1 + x^2 \sum_{d,r \geq 1} a(d,r,q^r) \GV_d Q^{dr}
+ x^3 \sum_{d,r \geq 1} b(d,r,q^r) \GV_d Q^{dr} .
\end{gather}
\end{Conjecture}

It is interesting to observe that the right hand side of~\eqref{Jquintic}
is a meromorphic function of $q$ with poles at roots of unity of bounded
order~3. In Section~\ref{sec.algorithm} we verify the above conjecture
modulo $O\big(Q^7\big)$ by an explicit calculation. Without doubt,
Conjecture~\ref{conj.1} concerns not only the quintic 3-fold, but
Calabi--Yau 3-folds with $h^{1,1}=1$ (there are plenty of those, see,
e.g.,~\cite{AvEvSZ}) and beyond. In contrast to the case of $\BC\BP^N$
(see~\eqref{eq.JCPN}) or the case of Fano manifolds,
the small $J$-function of the quintic is not hypergeometric. The above
conjecture was formulated independently by Jockers--Mayr~\cite[p.~10]{Jockers:qK}
and a comparison between their formulation and ours is given in
Section~\ref{sub.comparison}. Our conjecture also agrees with the
results of Jockers--Mayr presented in~\cite[Table~6.1]{Jockers:3d}.
Let us introduce the following multi-covering notation
\begin{gather*}
\GV^{(\gamma)}_n = \sum_{d|n} d^\gamma \GV_d.
\end{gather*}

Then, we have the following.
\begin{Corollary} \label{cor.q=0}
 We have
 \begin{gather}
 5 J(Q,0,0) = 5 + x^2 \sum_{n=1}^\infty n \GV^{(0)}_n Q^n + x^3
 \sum_{n=1}^\infty \big(n \GV^{(0)}_n+n^2\GV^{(-2)}_n \big) Q^n\nonumber \\
 \hphantom{5 J(Q,0,0)}{}
 = 5 + \big(2875 Q + 1224250 Q^2 + 951627750 Q^3 + 969872568500 Q^4 +
 \cdots \big) x^2\nonumber \\
 \hphantom{5 J(Q,0,0)=}{} +
\big(5750 Q + 1845000 Q^2 + 1268860000 Q^3 + 1212342581500 Q^4 +
 \cdots \big) x^3. \!\!\!\label{eq.Jq=0}
\end{gather}
\end{Corollary}
The above corollary reproduces the invariants of
equations~\eqref{012-016}. To extract them, let
$[J(Q,q,0)]_{x^\a}$ denote the coefficient of $x^\a$ in
$J(Q,q,0)$. The next corollary is proven in
Section~\ref{sub.extract}.

\begin{Corollary} \label{cor.numbers}
We have
\begin{gather}\label{ELqnum}
\sum_{d \geq 1}
\left\la \frac{\Phi_\a}{1-qL} \right\ra_{0,1,d} Q^d
= \begin{cases}
 - 5 [J(Q,q,0)]_{x^2} + 5 [J(Q,q,0)]_{x^3} & \text{if} \quad \a=0,
 \\
 \hphantom{-} 5 [J(Q,q,0)]_{x^2} & \text{if} \quad \a=1,
 \\
 \hphantom{-} 0 & \text{if} \quad \a=2,3 .
 \end{cases}
 \end{gather}
 Setting $q=0$, it follows that
\begin{gather}
\sum_{d \geq 1} \la 1 \ra_{0,1,d} Q^d
 = \sum_{n=1}^\infty\! n^2 \GV^{(-2)}_n Q^n =
 2875 Q + 620750 Q^2\! + 317232250 Q^3\! + 242470013000 Q^4 \nonumber\\
\hphantom{\sum_{d \geq 1} \la 1 \ra_{0,1,d} Q^d=}{}
 + 229305888959500 Q^5 + 248249743392434250 Q^6 + \cdots\label{eq.JJq=0}
 \end{gather}
 matching with equations~\eqref{012-016} $($being
 the generating series of the K-theoretic versions of
 the GV-invariants, given in the second page and in {\rm \cite[Table~6.1]{Jockers:3d})}, as well as
\begin{gather}
\sum_{d \geq 1} \la \Phi_1 \ra_{0,1,d} Q^d = \sum_{n=1}^\infty \! n  \GV^{(0)}_n Q^n
 =
 2875 Q + 1224250 Q^2\! + 951627750 Q^3\! + 969872568500 Q^4 \nonumber\\
\hphantom{\sum_{d \geq 1} \la \Phi_1 \ra_{0,1,d} Q^d =}{}
 +1146529444452500 Q^5 + 1489498454615043000 Q^6 + \cdots.\nonumber\label{eq.JJq=02}
\end{gather}
\end{Corollary}

\subsection[The linear q-difference equation for the quintic]{The linear $\boldsymbol{q}$-difference equation for the quintic}
\label{sub.qunticqdiff}

In this section we give an explicit formula for the small linear $q$-difference
equation for the quintic, assuming Conjecture~\ref{conj.1}. A key feature
of this formula is that the coefficients of this equation are analytic
(as opposed to polynomial) functions of $Q$ and $q$.
The small $J$-function $J(Q,q,0)$, viewed as a vector in the vector space
$K(X)$, forms the first column of the matrix $T(Q,q,0)$
of fundamental solutions of the small linear $q$-difference equation
in the basis $\big\{1,x,x^2,x^3\big\}$ of $K(X)$. The formula~\eqref{Jquintic} for the small
$J$-function and that fact that it is a cyclic vector of the linear
$q$-difference equation allows us to reconstruct the matrix $A(Q,q,0)$. See
also~\cite[Theorem~1.1, Lemma~3.3]{Iritani:rec}. To do so,
let us introduce some useful notation. If $f=f(d,r,q) \in \BQ(q)$ we denote
\begin{gather*}
[f] = \sum_{d,r \geq 1} f(d,r,q^r) \GV_d Q^{dr} .
\end{gather*}
With this notation, equation~\eqref{Jquintic} becomes
\[
\frac{1}{1-q}J(Q,q,0) = 1 + [a]x^2+[b]x^3
= \begin{pmatrix} 1 \\ 0 \\ [a] \\ [b] \end{pmatrix}
\]
in the basis $\big\{1,x,x^2,x^3\big\}$ of $K(X)$, where $a$, $b$ are given
by~\eqref{eq.a-eq.b}.
Further, we denote $(Ef)(d,r,q)\allowbreak =q^{d}f(d,r,q)$, and define
\begin{gather} \label{eq.cd}
 5c = \pi_+((1-E)a),\qquad
 5d = \pi_+(Ea+(1-E)b) ,
\end{gather}
with projections $\pi_\pm\colon \calK(q) \to \calK_\pm(q)$ given in
Section~\ref{sub.notation}. Explicitly, we have
\begin{gather*}
 5 c(d,r,q) = \frac{d^2}{1-q}, \\
 5 e(d,r,q) = \frac{d r}{1-q}-\frac{d(dq+q-d)}{(1-q)^2} .
\end{gather*}

Recall the $T$ matrix from~\cite[Proposition~2.3]{Iritani:rec} which is
a fundamental solution of the linear $q$-difference equation, and whose
first column is~$J$. The proof of the next theorem and its corollary is
given in Section~\ref{sub.thm1}.

\begin{Theorem} \label{thm.1}
 Conjecture~{\rm \ref{conj.1}} implies that the small $T$-matrix of the quintic
 is given by
 \begin{gather} \label{Tquintic}
 T(Q,q,0)=\begin{pmatrix}
 1 & 0 & 0 & 0 \\
 0 & 1 & 0 & 0 \\
 [a] & [c] & 1 & 0 \\
 [b] & [e] & 0 & 1
 \end{pmatrix}
\end{gather}
and the small $A$-matrix of the linear $q$-difference equation is
given by
 \begin{gather} \label{Aquintic}
 A=I-D^{\rm T}, \qquad
 D(Q,q,0)=\begin{pmatrix}
 0 & 1 & [a-c-Ea] & [b-e+Ea-Eb] \\
 0 & 0 & 1+[c-Ec] & [e+Ec-Ee] \\
 0 & 0 & 0 & 1 \\
 0 & 0 & 0 & 0
 \end{pmatrix} .
 \end{gather}
\end{Theorem}
Note that the entries of $5D(Q,q,0)$ are in $\BZ[[Q]][q]$ and given explicitly
in equations~\eqref{eq.d13-eq.d24} below. Let us denote by
$c_{ttt}(Q,q,t)=5 D_{2,3}(Q,q,t)$, where $D_{i,j}$ denotes the $(i,j)$-entry
of the matrix $D$. In other words, we have
\begin{gather*}
\begin{split}
& c_{ttt}(Q,q) = 5 + \sum_{d,r \geq 1} d^2 \frac{1-q^{dr}}{1-q^r} \GV_d Q^{dr}\\
& \hphantom{c_{ttt}(Q,q)}{}
= \sum_{d =1}^\infty d^2 \GV_d
\big(\Li_0\big(Q^d\big) + \Li_0\big(qQ^d\big) + \dots + \Li_0\big(q^{d-1}Q^d\big)\big),
\end{split}
\end{gather*}
where $\Li_s$ denotes the $s$-polylogarithm function
$\Li_s(z)=\sum_{d \geq 1} z^d/d^s$.
Recall the genus 0 generating series (minus its quadratic part)
of the quintic~\cite{Candelas:pair,Cox-Katz}
\begin{gather*}
\calF(Q) = \sum_{n=1}^\infty \GW_n Q^n =
\frac{5}{6} (\log Q)^3 + \sum_{d=1}^\infty \GV_d \Li_3\big(Q^d\big)
\end{gather*}
and its third derivative
\begin{gather}\label{eq.cttt}
c_{ttt}(Q) =(Q \partial_Q)^3 \calF(Q)=
5+\sum_{d=1}^\infty d^3 \GV_d \Li_0\big(Q^d\big),
\end{gather}
where $\partial_Q=\partial/\partial_Q$.

The next corollary gives a second relation between the $q=1$ limit of
quantum K-theory and quantum cohomology.

\begin{Corollary} \label{cor.ctttq}
The function $c_{ttt}(Q,q) \in \BZ[[Q]][q]$ is a $q$-deformation of the
Yukawa coupling $($i.e., $3$-point function$)$ $c_{ttt}(Q)$ in~\eqref{eq.cttt}.
Indeed, we have
\begin{gather*}
c_{ttt}(Q,1)=c_{ttt}(Q), \qquad 5 D_{2,3}(Q,q,0) = c_{ttt}(Q,q) .
\end{gather*}
\end{Corollary}

Thus, the $q$-difference equation of the quantum K-theory of the quintic
is a $q$-deformation of the well-known Picard--Fuchs equation of the quintic.

Let us abbreviate the four nontrivial entries of $D(Q,q,0)$ by
\begin{gather*}
\alpha = D_{1,3}, \qquad \beta = D_{1,4}, \qquad
\gamma = D_{2,3}, \qquad \delta = D_{2,4} .
\end{gather*}

\begin{Lemma}[{\cite[equations~(8.22) and (8.23)]{Jockers:3d}}] \label{lem.D}
 The linear $q$-difference equation
 \begin{gather*}
 \Delta \begin{pmatrix} y_0 \\ y_1 \\ y_2 \\ y_3 \end{pmatrix}
 =
 \begin{pmatrix}
 0 & 1 & \alpha & \beta \\
 0 & 0 & \gamma & \delta \\
 0 & 0 & 0 & 1 \\
 0 & 0 & 0 & 0
 \end{pmatrix} \begin{pmatrix} y_0 \\ y_1 \\ y_2 \\ y_3 \end{pmatrix}
 \end{gather*}
 $($where $\Delta=1-E)$ is equivalent to the equation
 \begin{gather} \label{eq.calLD}
 \calL y_0 =0, \qquad \calL =
 \Delta \left(1+ \Delta
 \frac{\delta+E\alpha+\Delta\beta}{\gamma+\Delta\alpha}\right)^{-1}
 \Delta (\gamma+\Delta\alpha)^{-1} \Delta^2 .
 \end{gather}
\end{Lemma}

We now discuss the $q \to 1$ limit, using the realization of the $q$-commuting
operators $E={\rm e}^{hQ\pt_Q}$ and $Q$ which act on a function $f(z,h)$ by
 \begin{gather*}
 (Ef)(z,h)=f(z+h,h), \qquad (Qf)(z,h)={\rm e}^zf(z,h), \qquad EQ={\rm e}^hQE,
 \end{gather*}
 where $Q={\rm e}^z$ and $q={\rm e}^h$. Then, in the limit $h \to 0$, the operator
 $\calL$ is given by
 \begin{gather} \label{eq.calL1}
 \calL(\Delta,Q,q)= \frac{1}{\gamma(Q,1)} \Delta^4
 + \partial^2_z \frac{1}{\gamma(Q,1)} \partial^2_z h^4 +O\big(h^5\big),
 \end{gather}
 where $5\gamma(Q,1)=c_{ttt}(Q,1)$. Thus, the coefficient of $h^4$ is
 the Picard--Fuchs equation of the quintic, whereas the coefficient of
 $h^0$ (the analogue of the AJ conjecture) is a line $(1-E)^4=0$ with
 multiplicity~4, punctured at the zeros of $\gamma(Q,1)=0$.
 It is not clear if one can apply topological recursion on such a degenerate
 curve.

\section{A review of quantum K-theory}\label{sub.review}

\subsection{Notation}\label{sub.notation}

In this section we collect some useful notation that we use throughout
the paper. For a smooth projective variety $X$, let $K(X)=K^0(X;\BQ)$ denote
the Grothendieck group of topological complex vector bundles with rational
coefficients.

Although we will not use it, the Chern class map induces a rational
isomorphism of rings
\begin{gather*}
\ch\colon \ K(X) \otimes \BQ \to H^{\rm ev}(X,\BQ)
\end{gather*}
between K-theory and even cohomology.
The ring $K(X)$ has a basis $\{\Phi_\a\}$ for $\a=0,\dots,N$ such
that $\Phi_0=1=[\calO_X]$ is the identity element. There is a nondegenerate
pairing on $K(X)$ given by $(E,F) \in K(X) \otimes K(X) \mapsto
 \chi(E\otimes F)$, where
\begin{gather*}
 \chi(E) = \int_X \ch(E)\mathrm{td}(X)
\end{gather*}
is the holomorphic Euler characteristic of $E$.
Let $\{\Phi^\a\}$ denote the dual basis of $K(X)$ with respect to the above
pairing. Let $\{P_1,\dots, P_r\}$ denote a collection of vector bundles
whose first Chern class forms a nef integral basis of $H^2(X,\BZ)/\text{torsion}$,
and let $Q=(Q_1,\dots,Q_r)$ be the collection of Novikov variables dual
to $(P_1,\dots,P_r)$.

The vector space $\calK(q)=\BQ(q)$ admits a symplectic form
\begin{gather*}
\omega(f,g) = (\mathrm{Res}_{q=0} +
\mathrm{Res}_{q=\infty}) \left(
 f(q) g\big(q^{-1}\big) \frac{\mathrm{d}q}{q}\right)
\end{gather*}
and a splitting
\begin{gather*}
\calK(q)=\calK_+(q) \oplus \calK_-(q)
\end{gather*}
(with projections $\pi_\pm\colon \calK(q) \to \calK_\pm(q)$) into a direct sum
of two Lagrangian susbpaces $\calK_+(q)=\BQ\big[q^{\pm 1}\big]$ and
$\calK_-(q)$, the space of reduced functions of $q$, i.e., rational
functions of negative degree which are regular at $q=0$.

\subsection{Reconstruction theorems for quantum K-theory}\label{sub.reconstruct}

In our paper we will focus exclusively on the genus 0 quantum K-theory
of $X$ (i.e., $g=0$ in~\eqref{EL}).
The collection of integers~\eqref{EL} can be encoded in several generating
series. Among them is the primary potential
\begin{gather*}
\calF_X(Q,t) = \sum_{d, n} \la t,\dots,t \ra_{0,n,d} \frac{Q^d}{n!}
\in \BQ[[Q,t]]
\end{gather*}
(where the summation is over $d \in \text{Eff}(X)$ and $n \geq 0$),
the $J$-function
\begin{gather*}
J_X(Q,q,t) = (1-q)\Phi_0 + t + \sum_{d, n} \sum_\a
\!\left\la t,\dots,t, \frac{\Phi_\a}{1-qL} \right\ra_{0,n+1,d}\! \Phi^\a Q^d
\in K(X) \otimes \calK(q)[[Q,t]]
\end{gather*}
(where $\{\Phi_\a\}$ is a basis for $K(X)$ for $\a=0,\dots,N$ with $\Phi_0=1$),
and the $T$ matrix $T_{\a,\b}(Q,q,t) \in \End(K(X)) \otimes \calK(q)[[Q,t]]$
and its inverse, whose definition we omit
but may be found in~\cite[Section~2]{Iritani:rec}.
We may think of $\calF_X$, $J_X(Q,q,t)$ and $T(Q,q,t)$
as scalar-valued, vector-valued and matrix-valued invariants, respectively.
$J_X(Q,q,t)$ specializes to $J_X(Q,q,0)$ when
$t=0$ and specializes to $\calF_X(Q,t)$ when $\a=0$ (as follows from
the string equation). Also, the $\a=0$ column of $T$ is~$J_X$.

There are several reconstruction theorems that determine all the
invariants~\eqref{EL} from others. In~\cite[Theorem~1.1]{Iritani:rec}, it
was shown that the small $J$-function $J_X(Q,q,0)$ uniquely determines
the $J$-function $J_X(Q,q,t)$, the primary potential $\calF_X(Q,t)$ and
the integers~\eqref{EL} (with $g=0$), under the assumption that $K(X)$
is generated by line bundles. In~\cite{Givental:equivariantVIII} it was
shown (under the same assumption on $X$) that the small $J$-function
$J_X(Q,q,0)$ reconstructs a permutation-equivariant version
of the quantum K-theory of $X$. This theory was introduced by Givental
in~\cite{Givental:equivariantVIII}, where this
theory takes into account the action of the symmetric groups $S_n$ on the moduli
spaces $\overline{\calM}^{X,d}_{g,n}$ that permutes the marked points.
The $J$ function of the permutation-equivariant quantum K-theory of $X$ takes
values in the ring $K(X) \otimes \calK(q) \otimes \Lambda[[Q]]$ where
$\Lambda$ is the ring of symmetric functions in infinitely many
variables~\cite{Macdonald}. $K(X)$, $\BQ[[Q]]$ and $\Lambda$ are
$\lambda$-rings with Adams operations $\psi^r$, so is their tensor product.
Moreover, the small $J$ function of the permutation-equivariant quantum
K-theory of $X$ agrees with the small $J$-function $J_X(Q,q,0)$ of the (ordinary)
genus 0 quantum K-theory of $X$.
According to a reconstruction theorem of Givental~\cite{Givental:equivariantVIII}
one can recover all genus zero permutation-equivariant K-theoretic GW invariants
of a projective manifold $X$ (under the mild assumption that the ring $K(X)$ is
generated by line bundles) from any point $t^*$ on their K-theoretic
Lagrangian cone via an explicit flow. In fortunate situations (that apply to
the quintic as we shall see below), one is given a value
$J_X(Q,q,t^*) \in K(X) \otimes \calK(q)[[Q]] \subset
K(X) \otimes \calK(q) \otimes \Lambda[[Q]]$
and $t^* \in K(X) \otimes \calK_+(q)[[Q]]$ (e.g., $t^*=0$), in which
case there exists a unique $\ve(x,Q,q) \in K(X) \otimes Q \calK_+(q)[[Q]]$
such that for all $t$
\begin{gather}\label{eq:flow}
J_X(Q,q,t) = \exp\left(\sum_{r \geq 1}\frac{\psi^r (\ve((1-x)E,Q,q))}{r(1-q^r)}
\right) J_X(Q,q,t^*) \in K(X) \otimes \calK(q)[[Q]],
\end{gather}
where $E$ is the operator that shifts $Q$ to $qQ$. The key point here is that
the coefficients of $\ve(x,Q,q)$ (for each power of $Q$ and $x$) are in the
subspace $\calK_+(q)$ of $\calK(q)$ whereas the corresponding coefficients
of $J_X(Q,q,t)$ are in the complementary subspace $\calK_-(q)$ of~$\calK(q)$.
Another key point is that although the above formula a priori is an equality
in the permutation-equivariant quantum K-theory, in fact it is an equality
of the ordinary quantum K-theory when $\ve$ is independent of $\Lambda$.

It follows that a single value $J_X(Q,q,t^*) \in K(X) \otimes \calK(q)[[Q]]$
uniquely determines $t^*$ as
well as the small J-function $J_X(Q,q,0)$, which in turn determines the
permutation-equivariant $J$-function $J_X(Q,q,t)$ for all $t$ via~\eqref{eq:flow}.

\subsection[A special value for the J-function of the quintic]{A special value for the $\boldsymbol{J}$-function of the quintic}\label{sub.specialJ}

For concreteness, we will concentrate on the case $X$ of the quintic.
To use the above formula~\eqref{eq:flow} we need the value $J_X(Q,q,t^*)$
at some point $t^*$. Such a value was given by Givental
in~\cite[p.~11]{Givental:equivariantV} and by Tonita in \cite[Theorem~1.3 and Corollary~6.8]{Tonita:quintic} who proved that if
$J_d$ denotes the coefficient of $Q^d$ in $J_{\BC\BP^4}(Q,q,0)$ given
in~\eqref{eq.JCPN}, then
\begin{gather}\label{eq.JXt*}
 I_{\calO(5)}(Q,q) =
 \sum_{d = 0}^\infty J_d \big((1-x)^5q;q\big)_{5d} Q^d
 = (1-q) \sum_{d = 0}^\infty \frac{\big((1-x)^5q;q\big)_{5d}}{((1-x)q;q)_{d}^5}Q^d
\end{gather}
lies on the K-theoretic Lagrangian cone of the quintic~$X$. This means
that if $\iota\colon X \to \BC\BP^4$ is the
inclusion, and $\iota^*\colon K(\BC\BP^4) = \BQ[x]/(x^5) \to K(X)=\BQ[x]/(x^4) $
is the induced map (sending $x \bmod x^5$ to $x \bmod x^4$),
there exists a~$t^*$ such that $\iota^* I_{\calO(5)}(Q,q) = J_X(Q,q,t^*)$.
In other words, we have
\begin{gather}\label{eq.JXt*-2}
J(Q,q,t^*) = (1-q)\sum_{d=0}^\infty \frac{((1-x)q;q)_{5d}}{
 ((1-x)q;q)_d^5} Q^d
\in K(X)\otimes \calK(q)[[Q]].
\end{gather}

Interestingly, the above formula has been interpreted by
Jockers and Mayr as an example of the 3d-3d correspondence of gauged linear
$\sigma$-models~\cite{Jockers:3d}. More precisely, the disk partition
function of a 3d gauged linear $\sigma$-model is a one-dimensional (so-called
vortex) integral whose integrand is a ratio of infinite Pochhammer symbols.
A residue calculation then produces the $q$-hypergeometric series~\eqref{eq.JXt*}.

\section[The flow of the J-function]{The flow of the $\boldsymbol{J}$-function}\label{sec.algorithm}

\subsection{Implementing the flow}\label{sub.implement}

In this section we explain how to obtain a formula for the small $J$-function
of the quintic (one power of $Q$ at a time) using formula~\eqref{eq.JXt*}
and the flow~\eqref{eq:flow}.
Observe that the coefficients of $q$ in the function $J(Q,q,t^*)$
given in~\eqref{eq.JXt*-2} are not in $\calK_-(q)$. For instance,
\begin{gather*}
\text{coeff}\left(\frac{1}{1-q}J(Q,q,t^*), x^0\right)=
\sum_{d=0}^\infty \frac{(q;q)_{5d}}{(q;q)_d^5} Q^d
\end{gather*}
is a power series in $Q$ whose coefficients are in $\calK_+(q)$
(and even in $\BN[q]$) and not in $\calK_-(q)$.
Note also that the function $J(Q,q,t^*)$ satisfies a 24th order (but
\emph{not} a 4th order) linear $q$-difference equation with polynomial
coefficients. This is discussed in detail in Section~\ref{sub.frob} below.

To find $J(Q,q,0)$ from $J(Q,q,t^*)$, we need to apply a flow
operator~\eqref{eq:flow}. To state the theorem, recall that
$K(X) \otimes \calK(q)[[Q]]$ is a $\lambda$-ring with Adams operations
$\psi^{(r)}$ given by combining the usual Adams operations in K-theory with
the replacement of $Q$ and $q$ by $Q^r$ and $q^r$. More precisely, for
a positive natural number $r$, we have
\begin{gather*}
\psi^{(r)}\colon \ K(X) \otimes \calK(q)[[Q]] \to K(X) \otimes \calK(q)[[Q]],
\psi^{(r)}\big((1-x)^i f(q) Q^j\big)=(1-x)^{ri} f(q^r) Q^{rj}
\end{gather*}
for $f(q) \in \calK(q)$ and natural numbers $i$, $j$ and $x$ as in~\eqref{eq.KX}.
Recall that the plethystic
exponential of $f(x,Q,q) \in K(X) \otimes \calK(q)[[Q]]$ (with $f(x,0,q)=0$)
is given by
\begin{gather*}
\Exp(f) = \exp\left(\sum_{r=1}^\infty \frac{\psi^{(r)}(f)}{r}\right).
\end{gather*}
It is easy to see that when $f$ is small (i.e., $f(x,0,q)=0$), then
$\Exp(f) \in K(X) \otimes \calK(q)[[Q]]$ is well-defined.
Let $E$ denote the $q$-difference operator that shifts $Q$ to $qQ$, as
in~\eqref{eq.cd}. By slight abuse of notation, we denote
\begin{align}
E \colon \ K(X) \otimes \calK(q)[[Q]] & \to K(X) \otimes \calK(q)[[Q]],\nonumber\\
  E\big((1-x)^i f(Q) Q^j\big)& =(1-x)^i f(qQ) Q^j .\label{eq.E}
\end{align}
Throughout the paper, the operators $E$ and $Q$ will act on a function $f(Q,q)$ by
\begin{gather}\label{eq.EQf}
(E f)(Q,q)=f(qQ,q), \qquad (Q f)(Q,q) = Q f(Q,q), \qquad EQ=qQE .
\end{gather}
The theorem of Givental--Tonita asserts that there
exists a unique
\[ \ve(x,Q,q) \in K(X) \otimes Q \calK_+(q)[[Q]]\] such that
\begin{gather}\label{eq.Je}
\Exp\left(\frac{\ve((1-x)E,Q,q)}{1-q}\right)
J(Q,q,t^*) \in K(X) \otimes \calK_-(q)[[Q]]
\end{gather}
and then, the left hand side of the above equation is
$J(Q,q,0)$.
Equation~\eqref{eq.Je} is a non-linear fixed-point equation for $\ve$
that has a unique solution that may be found working on one $Q$-degree
at a time. Indeed, we can write
\[
\ve(x,Q,q)=\sum_{k=1}^\infty \ve_k(x,q)Q^k, \qquad
\ve_k(x,q)=\sum_{\ell=0}^3 \sum_{k=1}^\infty \ve_{k,\ell}(q) x^\ell Q^k .
\]
Then for each positive integer number $N$ we have
\begin{gather*}
\pi_+ \left( \exp\left(\sum_{r=1}^N \sum_{\ell=0}^3 \sum_{k=1}^N
 \frac{\psi^{(r)} \ve_{k,\ell}(q)}{r(1-q^r)} Q^{rk}((1-x)E)^{\ell r}\right)
J(Q,q,t^*) \right) =0.
\end{gather*}
Equating the coefficient of each power of $x^i$ for $i=0,\dots,3$ to zero
in the above equation, we get a system of four inhomogeneous linear equations
with unknowns $(\ve_{N,0},\dots,\ve_{N,3})$ (with coefficients polynomials
in $\ve_{N',\ell'}$ for $N'<N$), with a unique solution in the field $\calK(q)$.
A further check (according to Givental--Tonita's theorem) is that the unique
solution lies in $\calK_+(q)$, and even more, in our case we check that it lies
in $\BQ[q]$.
Once $\ve_{N'}(x,q)$ is known for $N' \leq N$, equation~\eqref{eq.Je} allows
us to compute $J_d(q)$, where
\begin{gather*}
J(Q,q,0)=\sum_{d=0}^\infty J_d(q) Q^d .
\end{gather*}
For instance, when $N=1$ we have
\begin{gather*}
 \ve_{1,0}(q) =
1724 + 572 q - 625 q^2 - 1941 q^3 - 3430 q^4 - 4952 q^5 - 6223 q^6 -
 6755 q^7 - 6184 q^8 \\
 \hphantom{\ve_{1,0}(q) =}{} - 4690 q^9 - 2747 q^{10} - 969 q^{11},
 \\
 \ve_{1,1}(q) =
-4600 - 1140 q + 2485 q^2 + 6520 q^3 + 11140 q^4 +
 15890 q^5 + 19860 q^6 + 21490 q^7 \\
 \hphantom{\ve_{1,1}(q) =}{}
 + 19630 q^8 + 14860 q^9 +
 8690 q^{10} + 3060 q^{11},
 \\
 \ve_{1,2}(q) =4025 + 555 q - 3115 q^2 - 7255 q^3 -
 12055 q^4 - 17020 q^5 - 21175 q^6 - 22850 q^7 \\
 \hphantom{\ve_{1,2}(q) =}{}
 - 20830 q^8 -
 15740 q^9 - 9190 q^{10} - 3230 q^{11},
 \\
 \ve_{1,3}(q) =-1150 + 10 q +
 1250 q^2 + 2670 q^3 + 4340 q^4 + 6080 q^5 + 7540 q^6 + 8120 q^7 \\
 \hphantom{\ve_{1,3}(q) =}{} +
 7390 q^8 + 5575 q^9 + 3250 q^{10} + 1140 q^{11},
\end{gather*}
and, consequently, we find that
\begin{gather*}
J_0(q) = 1 - q ,\\
 J_1(q) = -\frac{575 x^2}{-1 + q} - \frac{1150 (-1 + 2 q) x^3}{(-1 + q)^2}
\end{gather*}
in agreement with~\cite[eqation~(6.38)]{Jockers:3d}.
Continuing our computation, we find that
\begin{gather*}
 J_2(q) = -\frac{25 \big(9794 + 19496 q + 9725 q^2\big) x^2}{(-1 + q) (1 + q)^2}
\\ \hphantom{J_2(q) =}{} -
\frac{50 \big({-}7380 - 9748 q + 14760 q^2 + 29244 q^3 + 12139 q^4\big) x^3}{
 (-1 + q)^2 (1 + q)^3}
\end{gather*}
and
\begin{gather*}
J_3(q) = -\frac{25 \big(7613022 + 15225906 q + 22838859 q^2 + 15225860 q^3 +
 7612953 q^4\big) x^2}{(-1 + q) \big(1 + q + q^2\big)^2}
 \\
 \hphantom{J_3(q) =}{}
 - \frac{50}{(-1 + q)^2 \big(1 + q + q^2\big)^3}\big({-}5075440 - 7612953 q - 7612953 q^2 + 10150880 q^3\\
 \hphantom{J_3(q) =}{}
 + 22838859 q^4 + 30451812 q^5 + 17763442 q^6 +  7612953 q^7 \big) x^3.
\end{gather*}
Two further values of $J_d(q)$ for $d=4,5$ were computed but are
too long to be presented here. Based on this data, we guessed the
formula for $J(Q,q,0)$ given in~\eqref{Jquintic}. Finally, we computed~$J_6(q)$ and found that it is in agreement with out predicted
formula~\eqref{Jquintic}.

\subsection[Extracting quantum K-theory counts from the small J-function]{Extracting quantum K-theory counts from the small $\boldsymbol{J}$-function}\label{sub.extract}

In this section we give a proof of Corollary~\ref{cor.numbers}
for the quintic $X$. Recall that $K(X)$ from equation~\eqref{eq.KX}
has basis $\Phi_\a=x^\a$ for $\a=0,1,2,3$ with $x^4=0$ and inner
product
\begin{gather}\label{innerproduct}
 (\Phi_a,\Phi_b) = \int_X \Phi_a \Phi_b \mathrm{td}(X)
 =
 \begin{pmatrix}
 \hphantom{-}0 & \hphantom{-}5 & -5 & 5 \\
 \hphantom{-}5 & -5 & \hphantom{-}5 & 0 \\
 -5 & \hphantom{-}5 & \hphantom{-}0 & 0 \\
 \hphantom{-}5 & \hphantom{-}0 & \hphantom{-}0 & 0
 \end{pmatrix} .
\end{gather}
The dual basis $\{\Phi^a\}$ of $K(X)$ is given by
\begin{gather}
\label{pp1}
\Phi^0 = \tfrac{1}{5} \Phi_3, \qquad \Phi^1 = \tfrac{1}{5}(
\Phi_2+\Phi_3), \qquad \Phi^2 = \tfrac{1}{5}(\Phi_1+\Phi_2), \qquad
\Phi^3 = \tfrac{1}{5}(\Phi_0 + \Phi_1 - \Phi_3),
\end{gather}
and is related to the basis $\{\Phi_a\}$ by
\begin{gather*}
\Phi_0 = 5\big(\Phi^1-\Phi^2+\Phi^3\big),\qquad \Phi_1 = 5\big( \Phi^0 - \Phi^1
+ \Phi^2\big), \\
\Phi_2 = 5\big({-}\Phi^0 + \Phi^1\big), \qquad\Phi_3 = 5 \Phi^0 .
\end{gather*}
Substituting $\Phi^\a$ as above in equation~\eqref{smallJ} and
collecting the powers of~$x^\a$, it follows that
\begin{gather*}
 [J(Q,q,0)]_1 = 1-q + \frac{1}{5} \sum_{d \geq 1}
 \left\la \frac{\Phi_3}{1-qL} \right\ra_{0,1,d} Q^d, \\
[J(Q,q,0)]_x = \frac{1}{5} \sum_{d \geq 1}
\left( \left\la \frac{\Phi_2}{1-qL} \right\ra_{0,1,d}
+ \left\la \frac{\Phi_3}{1-qL} \right\ra_{0,1,d} \right)Q^d, \\
[J(Q,q,0)]_{x^2} = \frac{1}{5} \sum_{d \geq 1}
\left( \left\la \frac{\Phi_1}{1-qL} \right\ra_{0,1,d}
+ \left\la \frac{\Phi_2}{1-qL} \right\ra_{0,1,d} \right)Q^d, \\
[J(Q,q,0)]_{x^3} = \frac{1}{5} \sum_{d \geq 1}
\left( \left\la \frac{\Phi_0}{1-qL} \right\ra_{0,1,d}
+ \left\la \frac{\Phi_1}{1-qL} \right\ra_{0,1,d}
- \left\la \frac{\Phi_3}{1-qL} \right\ra_{0,1,d} \right)Q^d .
\end{gather*}
The above is a linear system of equations with unknowns
$\sum_{d \geq 1} \big\la \frac{\Phi_\a}{1-qL} \big\ra_{0,1,d} Q^d$
for $\a=0,1,2,3$. Solving the linear system combined
with equation~\eqref{eq.Jq=0}, gives~\eqref{ELqnum}.
Setting $q=0$ in~\eqref{ELqnum} and using Corollary~\ref{cor.q=0},
we obtain~\eqref{eq.JJq=0} and~\eqref{eq.JJq=02} and conclude the
proof of Corollary~\ref{cor.numbers}.

\subsection{A comparison with Jockers--Mayr}\label{sub.comparison}

In this section we give the details of the comparison of our
Conjecture~\ref{conj.1} with a conjecture of
Jockers--Mayr~\cite[p.~10]{Jockers:qK}.

To begin with, their $I_{QK}(t)$ is our $J(Q,q,t)$ and their
$I(0)$ in~\cite[equation~(7)]{Jockers:qK} is our $J(Q,q,0)$. They drop the index
QK later on. From~\cite[equation~(4)]{Jockers:qK} it follows that they are
working in the same basis $\Phi_\alpha=x^a$, $\alpha=0,1,2,3$, as we
are. Furthermore, the inner product on~$K(X)$ \cite[equation~(6)]{Jockers:qK}
agrees with the one given in equation~\eqref{innerproduct} with dual
basis $\{\Phi^a\}$ of $K(X)$ given in~\eqref{pp1}.
By~\cite[equation~(8)]{Jockers:qK}, specialized to the quintic, the function
$I(t)$ becomes $I(t) = 1 - q + t\Phi_1 + F^2(t) \Phi_2 + F^3(t) \Phi_3$.
Then, they define functions $F_A$ and $\hat F^A$ by writing
$\sum_A F^A\Phi_A = \sum_A \big(F_{A,cl}+\hat F_A\big)\Phi^A$, where
$\hat F_A(t) =
 \sum_{d>0} Q^d \big\langle\!\big\langle \frac{\Phi_A}{1-qL}\big\rangle\!\big\rangle_d$,
 cf.~\cite[equation~(9)]{Jockers:qK}, and $F_{A,cl}$ are ``constant'',
 i.e., independent of $Q$ and $t$. Note that only $F^2$, $F^3$ are nonzero
 which implies that only $F_0$, $F_1$ are nonzero.
 Their conjecture~\cite[p.~10]{Jockers:qK} can now be stated (in the case
 of the quintic) as follows~\cite[equation~(10)]{Jockers:qK}:
 \begin{gather*}
 \hat F_0 = p_2 + \frac{1}{(1-q)^2}[ (1-3q)\mathcal{F} + q t \mathcal{F}_1]_{t^{n>2}},\\
 \hat F_1 = p_{1,1} + \frac{1}{(1-q)}[\mathcal{F}_1]_{t^{n>1}},
 \end{gather*}
 where $p_2$, $p_{1,1}$, $\mathcal{F}$, $\mathcal{F}_1$ are certain explicitly
 given functions of $t$ and the Gopakumar--Vafa invariants
 $\GV_d$,~\cite[equations~(11) and (12)]{Jockers:qK}. Combining everything so far,
 their conjecture reads
 \begin{gather*}
 I(t) = 1 - q + t\Phi_1 + (F_{1,cl}+p_{1,1} +
 \frac{1}{(1-q)}[\mathcal{F}_1]_{t^{n>1}})\Phi^1 \\
 \phantom{I(t) =}+ (F_{0,cl}+p_2 + \frac{1}{(1-q)^2}[ (1-3q)\mathcal{F} + q t \mathcal{F}_1]_{t^{n>2}})\Phi^0.
 \end{gather*}
 We will not spell out
 these functions completely, but only their value at $t=0$ in order
 to compare it to our formulas. First, the brackets
 $[\dots]_{t^{n>1}}, [\dots]_{t^{n>2}}$ vanish for $t=0$. So we are
 left with $p_2$ and $p_{1,1}$~\cite[equation~(12)]{Jockers:qK}. Noting that
 $\sum_j d_jt_j = 0$ for $t=0$, these read
 \begin{gather*}
 \frac{1}{1-q} p_{1,1} |_{t=0} = \sum_{d>0} Q^d \sum_{r|d} \GV_{d/r} \frac{ d(1-q^r) + \frac{d}{r}q^r}{(1-q^r)^2},\\
 \frac{1}{1-q} p_{2} |_{t=0} = \sum_{d>0} Q^d \sum_{r|d} \GV_{d/r} \frac{ r^2(1-q^r)^2 -q^r(1+q^r)}{(1-q^r)^3} .
 \end{gather*}
 Next, we rewrite these sums so that they run over all values of $r$
 \begin{gather*}
 \frac{1}{1-q} p_{1,1} |_{t=0} = \sum_{d,r>0} Q^{dr} \GV_{r} \frac{ dr(1-q^r) + dq^r}{(1-q^r)^2},\\
 \frac{1}{1-q} p_{2} |_{t=0} = \sum_{d,r>0} Q^{dr} \GV_{r} \frac{ r^2(1-q^r)^2 -q^r(1+q^r)}{(1-q^r)^3} .
 \end{gather*}
 Hence,
 \begin{gather*}
 \frac{1}{1-q} p_{1,1} |_{t=0} = 5 \sum_{d,r>0} Q^{dr} \GV_{r} a(d,r,q^r),\\
 \frac{1}{1-q} p_{2} |_{t=0} = 5\sum_{d,r>0} Q^{dr} \GV_{r} \left(b(d,r,q^r)-a(d,r,q^r)\right) .
 \end{gather*}
 The appearance of the term involving $a(d,r,q^r)$ in the second
 equation is due to the change of basis $\Phi_2=5\big({-}\Phi^0+\Phi^1\big)$. This completes the compatibility
 of our conjecture and theirs.

\section[q-difference equations]{$\boldsymbol{q}$-difference equations}\label{sec.JA}

\subsection[The small q-difference equation of the quintic]{The small $\boldsymbol{q}$-difference equation of the quintic}\label{sub.thm1}

In this section we explain how Theorem~\ref{thm.1} follows from
Conjecture~\ref{conj.1}. We begin with a general discussion. Given a
collection of vector functions $f_j(Q,q) \in \BQ(q)[[Q]]^r$
for $j=1,\dots,r$ such that $\det(f_1|f_2|\dots|f_r)$ is not
identically zero, there is always a canonical linear $q$-difference
equation
\begin{gather*}
(E y)(Q,q)=A(Q,q) y(Q,q)
\end{gather*}
with fundamental solution set $f_1,\dots,f_r$, where $E$ is the shift
operator of equation~\eqref{eq.E} that replaces $Q$ by $qQ$. Indeed,
the equations $Ey_j = A y_j$ for $j=1,\dots,r$ are equivalent to the
matrix equation $ET=A T$ where $T=(f_1|f_2|\dots|f_r)$ is the fundamental
matrix solution, and inverting~$T$, we find that $A=(ET)^{-1}T$. This
can be applied in particular to the case of a collection $E^j g$ for
$j=0,\dots,r-1$ of a vector function $g(Q,q) \in \BQ(q)[[Q]]$ that satisfies
$\det\big(g|Eg|\dots|E^{r-1}g\big)$ is nonzero. Said differently, every vector function
$g(Q,q) \in \BQ(q)[[Q]]$ along with its $r-1$ shifts (generically) satisfies
a linear $q$-difference equation.

We will apply the above principle to the 4-tuple
$((1-x)E)^jJ(Q,q,0)/(1-q) \in K(X)\otimes \calK(q)[[Q]]$
for $j=0,\dots,3$ where $J(Q,q,0) \in K(X)\otimes \calK_-(q)[[Q]]$ is as in
Conjecture~\ref{conj.1}. However, notice that although the $q$-coefficients
of $J(Q,q,0)/(1-q)$ are in $\calK_-(q)$, this is no longer true for the shifted
functions $((1-x)E)^jJ(Q,q,0)/(1-q)$ for $j=1,2,3$. In that case, we need to
apply the Birkhoff factorization~\cite[App.A]{Givental-Reconstruction}
to the matrix
\begin{gather}\label{eq.MTU}
\frac{1}{1-q}\big(J|(1-x)EJ|((1-x)E)^2J|((1-x)E)^3J\big) =T U,
\end{gather}
where the $q$-coefficients of the entries of $T$ are in $\calK_-(q)$ and of
$U$ are in $\calK_+(q)$ (compare also with
Lemma~3.3 of~\cite[equation~(4)]{Iritani:rec}).
The existence and uniqueness of matrices $T$ and $U$
in the above equation follows from the fact that the left hand side of the
above equation is unipotent, and the proof is discussed in detail in
the above reference.

In our case, the choice
\begin{gather*}
T = \frac{1}{1-q}\pi_+\big(J|(1-x)EJ|((1-x)E)^2J|((1-x)E)^3J\big)
\end{gather*}
together with equation~\eqref{eq.MTU} implies that the $q$-coefficients
of the entries of $U$ are in $\calK_+(q)$. Equation~\eqref{Tquintic}
for the fundamental matrix $T$ follows from the fact that
\begin{gather*}
\pi_+\left(q^d\left(\frac{r^2}{1-q}-\frac{q+q^2}{(1-q)^3}\right)\right)
=\frac{-1 + 3 q - 4 q^2}{(1 - q)^3} + \frac{(-1 + d) (-1 - d + 3 q + d q)}{
 (1 - q)^2} + \frac{r^2}{1 - q},\\
\pi_+\left(q^d\left(\frac{r}{1-q}+\frac{q}{(1-q)^2}\right)\right)
=
\frac{-d + q + d q}{(1 - q)^2} + \frac{r}{1 - q}
\end{gather*}
valid for all positive natural numbers $d$ and $r$.

Having computed the fundamental matrix $T$~\eqref{Tquintic}, we
use~\cite[equation~(2)]{Iritani:rec}, with $P^{-1} q^{Q \partial_Q}$ replaced by
$1-(1-x)E$ to deduce the small A-matrix~\eqref{Aquintic}.

Explicitly, the four nontrivial entries of the matrix $D$ are given by
\begin{subequations}\label{eq.d13-eq.d24}
\begin{gather}
 \label{eq.d13}
 5(a-c-Ea)(d,r,q) = \frac{d \big({-}d + q + d q - q^{1 + d} + r - q r - q^d r + q^{1 + d} r\big)}{(1 - q)^2}, \\
 5(b-e+Ea-Eb)(d,r,q) = -\frac{q^2 \big(1 + 2 d + d^2 - r^2\big) + q \big(1 - 2 d - 2 d^2 + 2 r^2\big) }{(1 - q)^3} \nonumber\\
 \phantom{5(b-e+Ea-Eb)(d,r,q) =} + \frac{ d^2 - r^2+ q^d \big({-}q - q^2 + r^2 - 2 q r^2 + q^2 r^2\big)}{(1 - q)^3}, \label{eq.d14} \\
 \label{eq.d23}
 5(c-Ec)(d,r,q) = \frac{d^2 \big(1 - q^d\big)}{1 - q}, \\
 \label{eq.d24}
 5(e+Ec-Ee)(d,r,q) = -\frac{d \big({-}d + q + d q - q^{1 + d} - r + q r + q^d r - q^{1 + d} r\big)}{(1 - q)^2} .
\end{gather}
\end{subequations}

Note that the entries of $5D$ are in $\BZ[[Q]][q]$. Moreover, the values
when $q=1$ are given by
\begin{gather*}
 5(a-c-Ea)(d,r,1) =-\frac{1}{2} d^2 (1 + d - 2 r), \\
 5(b-e+Ea-Eb)(d,r,1) =-\frac{1}{6} d \big(1 + 3 d + 2 d^2 - 6 r^2\big), \\
 5(c-Ec)(d,r,q) = d^3,\\
 5(e+Ec-Ee)(d,r,1) = \frac{1}{2} d^2 (1 + d + 2 r).
\end{gather*}

As a further consistency check, note
that our matrix $D$ given in~\eqref{Aquintic} equals to the mat\-rix~$D$
of~\cite[equation~(8.21)]{Jockers:3d}.

Given the formula of~\eqref{Aquintic}, an explicit calculation shows that
the entries of $D$ are given by~\eqref{eq.d13-eq.d24}. This concludes the proof of Theorem~\ref{thm.1}.

\begin{proof}[Proof of Corollary~\ref{cor.ctttq}] It follows from equations~\eqref{eq.d23} and~\eqref{eq.cttt}.
\end{proof}

\begin{proof}[Proof of Lemma~\ref{lem.D}]We have
 \begin{gather*}
 \Delta y_0 = y_1 + \alpha y_2 + \beta y_3, \qquad
 \Delta y_1 = \gamma y_2 + \delta y_3,\qquad
 \Delta y_2 = y_3, \qquad
 \Delta y_3 =0.
 \end{gather*}
 The lemma follows by eliminating $y_1$, $y_2$, $y_3$ (one at a time)
 using the fact that
 \begin{gather*}
 E(fg)=(Ef)(Eg), \qquad
 \Delta(fg)=(\Delta f) g + f (\Delta g)- (\Delta f)(\Delta g) .
 \end{gather*}
 (which follows from $(Ef)(Q,q)=f(qQ,q)$ and $\Delta=1-E$).
 Indeed, we have
\[
 \Delta^2 y_0 = \Delta(\Delta y_0) = \Delta(y_1 + \alpha y_2 + \beta y_3)
 =(\gamma+\Delta\alpha)y_2 + (\delta+E\alpha+\Delta\beta)y_3 ,
\]
 and hence,
\[
 (\gamma+\Delta\alpha)^{-1} \Delta^2 y_0 = y_2 +
 \frac{\delta+E\alpha+\Delta\beta}{\gamma+\Delta\alpha} y_3 ,
\]
 and hence,
\[
 \Delta (\gamma+\Delta\alpha)^{-1} \Delta^2 y_0 =
 \left(1+\frac{\delta+E\alpha+\Delta\beta}{\gamma+\Delta\alpha} \right) y_3.
\]
 Applying $\Delta$ once again and using $\Delta y_3=0$ concludes the
 proof of equation~\eqref{eq.calLD}. Note that the notation is such that
 an operator $\Delta$ is applied to everything on the right hand side.

 The $q=1$ limit of $\calL(\Delta,Q,q)$ follows from
equation~\eqref{eq.calLD}, the fact that
 \begin{gather*}
 (\Delta f)(Q,q)|_{q=1}=(f(qQ,q)-f(Q,q))|_{q=1}=0
 \end{gather*}
 and Corollary~\ref{cor.ctttq}.
\end{proof}

\subsection[The Frobenius method for linear q-difference equations]{The Frobenius method for linear $\boldsymbol{q}$-difference equations}
\label{sub.frob}

In this section we discuss in detail the linear $q$-difference equation
satisfied by the function $J(Q,q,t^*)$ of~\eqref{eq.JXt*}. Recall the operators
$E$ and $Q$ that act on functions of $Q$ and $q$ by~\eqref{eq.EQf}.

Let
\begin{gather}\label{eq.frob1}
J(Q,q,x)=\sum_{n=0}^\infty a_n(q,x) Q^n = J_0(Q,q) + J_1(Q,q) x + \dots
\in \BQ(q)[[Q,x]],
\end{gather}
where $J_n(Q,q) \in \BQ(q)[[Q]]$ for all $n$ and
\begin{gather*}
 a_n(q,x)=\frac{\big({\rm e}^{5x}q;q\big)_{5n}}{\big({\rm e}^x q;q\big)_n^5},
 \end{gather*}
 where ${\rm e}^{ax}$ is to be understood as a polynomial in $x$ obtained as
 ${\rm e}^{ax}+O\big(x^4\big)$.

 The functions $J_n(Q,q)$ are given by series whose summand ia a
 $q$-hypergeometric function times a polynomial of $q$-harmonic functions.
 For example, we have
 \begin{gather*}
 J_0(Q,q) = \sum_{n=0}^\infty \frac{(q;q)_{5n}}{(q;q)_n^5} Q^n, \\
 J_1(Q,q) = \sum_{n=0}^\infty \frac{(q;q)_{5n}}{(q;q)_n^5}
 (1 + 5 H_{5n}(q) -5 H_n(q)) Q^n,
 \end{gather*}
 where $H_n(q)=\sum_{j=1}^n q^j/\big(1-q^j\big)$ is the $n$th $q$-harmonic number.
 Consider the 25-th order linear $q$-difference operator
 \begin{gather} \label{eq.L5}
 L_5(E,Q,q)=(1-E)^5 - Q \prod_{j=1}^5 \big(1-q^j E^5\big)
 \end{gather}
 with coefficients polynomials in $Q$ and $q$. Note that
 $L_5=(1-E)^5 - \prod_{j=1}^5 \big(1-q^{5-j} E^5\big)Q$, hence~$L_5$ factors as
 $1-E$ times a 24-th order operator.

 \begin{Lemma} \label{lem.frob1}
 With $J$ as in~\eqref{eq.frob1} and $L_5$ as in~\eqref{eq.L5}, we have
 \begin{gather*}
 L_5\big({\rm e}^x,Q,q\big) J = \big(1-{\rm e}^x\big)^5.
 \end{gather*}
\end{Lemma}

\begin{proof} It is easy to see that
\[
 \frac{a_{n}(q,x)}{a_{n-1}(q,x)} =
 \frac{\prod_{j=1}^5\big(1-{\rm e}^{5x}q^{5n-j}\big)}{\big(1-{\rm e}^x q^n\big)^5}.
\]
Hence,
\[
 \big(1-{\rm e}^xq^n\big)^5 a_n(q,x) Q^n = Q \prod_{j=1}^5\big(1-{\rm e}^{5x}q^{5n-j}\big)
 a_{n-1}(q,x) Q^{n-1}
\]
 and in operator form,
\[
 \big(1-{\rm e}^x E\big)^5 a_n(q,x) Q^n = Q \prod_{j=1}^5\big(1- q^{j} {\rm e}^{5x} E^5 \big)
 a_{n-1}(q,x) Q^{n-1}.
\]
 Summing from $n=1$ to infinity, we obtain that
\[
 \big(1-{\rm e}^x E\big)^5 (J-1) = Q \prod_{j=1}^5\big(1- q^{j} {\rm e}^{5x} E^5\big) J.
\]
 Since $\big(1-{\rm e}^x E\big)^5=\big(1-{\rm e}^x\big)^5$, the result follows.
\end{proof}

Note that the proof of Lemma~\ref{lem.frob1} implies that $J(Q,q,x)$
satisfies a 24-th order linear $q$-difference equation but this will no play
a role in our paper. Of importance is the fact that the 25-th order equation
$L_5f=0$ has a distinguished 5-dimensional space of solutions, given explicitly
by a $q$-version of the Frobenius method. Since this method is well-known
for linear differential equations, but less so for linear $q$-difference
equations, we give more details than usual. For additional discussion on
this method, see Wen~\cite{Wen}, and for references for the $q$-gamma and
$q$-beta functions, see De~Sole--Kac~\cite{Desole-Kac}.

First, we define an $n$-th derivative of an operator $P(E,Q,q)$ by
\begin{gather*}
P^{(n)}(E,Q,q) = \sum_{k=0^d} k^n c_k(Q,q) E^k, \qquad
 P(E,Q,q) = \sum_{k=0^d} c_k(Q,q) E^k .
\end{gather*}
In other words, we may write $P^{(n)}=(E \pt_E)^n (P)$.
\begin{Lemma}\label{lem.frob2}
 For a linear $q$-difference operator $P(E,Q,q)$ we have
\begin{gather}\label{eq.Pexp}
P({\rm e}^x E,Q,q) = \sum_{n=0}^\infty \frac{x^n}{n!}P^{(n)}(E,Q,q) .
\end{gather}
Moreover, for all natural numbers $n$ and a function $f(Q,q)$ we have
\begin{gather}\label{eq.Plog}
P((\log Q)^n f) = \sum_{k=0}^n\binom{n}{k} (\log Q)^{n-k} (\log q)^k P^{(n-k)} f.
\end{gather}
\end{Lemma}

\begin{proof} Equations~\eqref{eq.Pexp} and~\eqref{eq.Plog} are additive in $P$, hence it
 suffices to prove them when $P=E^a$ for a natural number $a$,
 in which case $(E^a)^{(n)}=a^n$ and both identities are clear.
\end{proof}

\begin{Lemma} \label{lem.frob3}
 Suppose $P(E,Q,q)$ is a linear $q$-difference operators with coefficients
 polynomials in $E$ and $Q$, and $J(Q,q,x) \in \BQ(q)[[Q,x]]$ is such
 that
 \begin{gather} \label{eq.PN}
 P({\rm e}^xE,Q,q) J(Q,q,x) = O\big(x^{N+1}\big)
 \end{gather}
 for some natural number $N$. Then,
 \begin{gather} \label{eq.frob3a}
 \sum_{k=0}^n \binom{n}{k} P^{(k)} J_{n-k}=0
 \end{gather}
 for $n=0,\dots,N$, where $J_k=\coeff\big(J(Q,q,x),x^k\big)$, and
 \begin{gather} \label{eq.frob3b}
 P f_n=0, \qquad f_n = \sum_{k=0}^n \binom{n}{k}
 (\log Q)^{n-k} (\log q)^k J_k
 \end{gather}
 for $n=0,1,\dots,N$. In other words, the equation $P f=0$ has $N+1$ distinguished solutions
 given by
 \begin{gather*}
 f_0 = J_0, \\
 f_1 = \log Q J_0 + \log q J_1, \\
 f_2 = (\log Q)^2 J_0 + 2 \log Q \log q J_1 + (\log q)^2 J_2, \\
\cdots\cdots\cdots\cdots\cdots\cdots\cdots\cdots\cdots\cdots\cdots\cdots\cdots\cdots\cdots
 \end{gather*}
\end{Lemma}

\begin{proof}
 Equation~\eqref{eq.frob3a} follows easily using~\eqref{eq.Pexp}
 and by expanding the left hand side of equation~\eqref{eq.PN}
 into power series in $x$ and equating the coefficient of $x^n$ with
 zero for $n=0,1,\dots,N$. Equation~\eqref{eq.frob3b} follows from
 equations~\eqref{eq.frob3a} and~\eqref{eq.Plog}, and induction on $n$.
\end{proof}

\section{Quantum K-theory versus Chern--Simons theory}\label{sec.CS}

There are several hints in the physics literature pointing to a deeper
relation between Quantum K-theory and Chern--Simons gauge theory (e.g., for
3-manifolds with boundary, such as knot complements), see for instance
in~\cite{holo-blocks,DGG2,Dimofte:vortex,Jockers:qK,Jockers:3d}
and in references therein. In this section we discuss and comment on
the $q$-difference equations in Chern--Simons theory, gauged linear
$\sigma$-models and Quantum K-theory. We will discuss three aspects
of this comparison:
\begin{itemize}\itemsep=0pt
\item[(a)] $q$-holonomic systems and their $q=1$ semiclassical limits,
\item[(b)] $\ve$-deformations,
\item[(c)] matrix-valued invariants.
\end{itemize}

We begin with the case of the Chern--Simons theory. The partition
function of Chern--Simons theory with compact (e.g., $\SU(2)$) gauge group
on a 3-manifold (with perhaps nonempty boundary) is given by a
finite-dimensional state-sum whose summand has as a building block
the quantum $n$-factorial. This follows from existence of an
underlying TQFT~\cite{RT,Tu:book, Wi:CS} which reduces the computation
of the partition function into elementary pieces. For the complement
of a knot $K$ in~$S^3$, the partition function recovers the colored Jones polynomial
of a knot which, in the case of~$\SU(2)$, is a sequence
$J_{K,n}(q) \in \BZ\big[q^{\pm}\big]$ of Laurent polynomials which can be
presented as a finite-dimensional sum whose summand has as a building
block the finite $q$-Pochhammer symbol~$(q;q)_n$. This ultimately boils
down to the entries of the $R$-matrix which are given for example
in~\cite{RT}.

On the other hand, Chern--Simons theory with complex (e.g., $\SL_2(\BC)$)
gauge group is not well-understood as a TQFT. However, the
partition function for a 3-manifold with boundary can be computed
by a finite-dimensional state-integral whose integrand has as a building
block Faddeev's quantum dilogarithm function~\cite{Fa}. The latter
is a ratio of two infinite Pochhammer symbols which form
a quasi-periodic function with two quasi periods. \big(Recall that the
Pochhammer symbol is $(x;q)_\infty=\prod_{j=0}^\infty \big(1-q^j x\big)$.\big)
These are the state-integrals studied in quantum Teichm\"{u}ller theory
by Kashaev et al.~\cite{AK:TQFT,AK:ICM,KLV} and in complex Chern--Simons
theory by Dimofte et al.~\cite{Dimofte:complexCS, Dimofte:perturbative}.

The appearance of $q$-holonomic systems in Chern--Simons
theory with compact/complex gauge group is a consequence of Zeilberger
theory~\cite{PWZ,WZ, Zeilberger} applied to finite-dimensional
state-sums/integrals whose summand/integrand has as a building block
the finite/infinite $q$-Pochhammer symbol. This is exactly how
it was deduced that the sequence of colored Jones polynomials
$J_{K,n}(q)$ of a knot satisfy a linear $q$-difference
equation $A_K\big(\hat L, \hat M,q\big) J_K=0$ (see~\cite{GL:qholo}),
where $\hat L$ and $\Hat M$ are $q$-commuting operators that act
on a sequence $f\colon \BN \to \BQ(q)$ by
\begin{gather*}
\big(\hat L f\big)(n)=f(n+1), \qquad \big(\hat M f\big)(n)=q^n f(n), \qquad LM=qML .
\end{gather*}
In the case of state-integrals, the existence of two quasi-periods
leads to a linear $q$- (and also $\tilde{q}$)-difference equation, where
$q={\rm e}^{2\pi {\rm i} h}$ and $\tilde{q}={\rm e}^{-2\pi {\rm i}/h}$.

It is conjectured that the linear $q$-difference equation of the
colored Jones polynomial essentially coincides with the one of
the state-integral, and that the classical $q=1$ limit (the so-called
AJ conjecture~\cite{Ga:AJ}) coincides with the $A$-polynomial
$A_K(L,M,1)$ of the knot.
The latter is the $\SL_2(\BC)$-character variety of the fundamental
group of the knot complement, viewed from the boundary
torus~\cite{CCGLS}. Finally, the semiclassical limit (the analogue
of~\eqref{eq.calL1} is given by
\begin{gather*}
A_K\big(\hat L, \hat M,q\big) = A_K(L,M,1) + D_K(z,\pt_z)h^s + O\big(h^{s+1}\big),
\end{gather*}
where $D_K(z,\pt_z)$ is a linear differential operator of degree $s$
where $s$ is the order of vanishing of $A_K(L,1,1)$ at $L=1$. This
order is typically $1$ (e.g., for the $4_1$, $5_2$, $6_1$ and more
generally all twist knots) but it is equal to $2$ for the $8_{18}$
knot.

We now come to the feature, namely an expected ``factorization''
of state-integrals into a finite sum of products of $q$-series
and $\tilde{q}$-series. This factorization is computed by an
$\ve$-deformation of~$q$- and $\tilde{q}$-hypergeometric series
that arise by applying the residue theorem to the state-integrals.
For a detailed illustration of this, we refer the reader
to~\cite{holo-blocks,GK:qseries} and~\cite{GZ:qseries}.

Our last discussed feature, namely a matrix-valued extension of the
Chern--Simons invariants with compact/complex gauge group
was recently discovered in two
papers~\cite{GZ:qmodularity, GZ:qseries}. More precisely, it
was conjectured and in some cases verified that the scalar valued
quantum knot invariants such as the Kashaev invariant~\cite{K95}
(an evaluation of the $n$-th colored Jones polynomial at $n$-th
roots of unity) and the Andersen--Kashaev state-integral~\cite{AK:TQFT}
admit an extension into a matrix-valued invariants. The rows and
columns are labeled by the set $\calP_M$ of $\SL_2(\BC)$
boundary-parabolic representations of $\pi_1(M)$. In the case of
a knot complement, the set $\calP_M$ can be thought of as the set of
branches of the $A$-polynomial curve above a point (where the
meridian has eigenvalues~1). Although the
corresponding vector space $R(M):=\BQ \calP_M$ with basis $\calP_M$ has
no ring structure known to us, it has a distinguished element
corresponding to the trivial $\SL_2(\BC)$-representation that
plays an important role. A ring structure $\BQ \calP_M$
might be defined as the Grothendieck group of an appropriate category
associated to flat connections on 3-manifolds with boundary, or
perhaps by contructing an appropriate logarithmic conformal field theory
use fusion rules will define the sought ring as suggested by Gukov.
Alternatively, the sought ring might be described in terms of
$\SL(2,\BC)$-Floer homology, suggested by Witten. Alternatively,
it might be described by the quantum K-theory of the mirror of the
local Calabi--Yau manifold $uv=A_M(x,y)$, (where $A_M$ is the
$A$-polynomial discussed above), suggested by
Aganagic--Vafa~\cite{Aganagic:2012jb}.

We now discuss the above features (a)--(c) that appear in the
3d-gauged linear $\s$-models and their 3d-3d correspondence
studied in detail in~\cite{holo-blocks,DGG2,DGG1,Dimofte:vortex,
 Jockers:qK,Jockers:3d} and references
therein. The $q$-holonomic aspect is still present since the
(so-called vortex) partition function is a finite-dimensional
integral whose integrand has as a building block the infinite Pochhammer
symbol (note however that $\tilde{q}$ does not appear). The second
aspect involving $\ve$-deformations is also present for the same
reason as in Chern--Simons theory. The third aspect is absent in general.

We finally discuss the above features in genus 0 quantum K-theory
of the quintic. The first aspect is different: the linear $q$-differential
equation has coefficients which are analytic (and not polynomial)
functions of $Q$ and $q$. The classical limit $q=1$ of
the linear $q$-difference equation of the quintic is given by
$\gamma(Q,1)^{-1}\Delta^4$~\eqref{eq.calL1} and this defines a
degenerate analytic curve in $\BC\times\BC$ that consists of a finite
collection of lines with coordinates $(\Delta,Q)$. On the other hand,
the semi-classical limit (i.e., the coefficient of $h^4$
in~\eqref{eq.calL1}) is the famous Picard--Fuchs equation of the quintic.
The second feature, the $\ve$-deformation for a nilpotent variable $\ve$
is encoded in the fact that $K(X)$ has nilpotent elements $x$.
The last feature is most interesting since the matrix-valued invariants
are encoded in $\End(K(X))$, where $K(X)$ is not just a rational vector
space, but a ring unit $1$. It follows that the linear $q$-difference
equations have not only a distinguished solution $J_X(Q,q,0)$ but
a basis of solutions parametrized by a basis
$\{\Phi_\a\}$ of $K(X)$.

Let us end our discussion with some questions on the colored Jones
polynomial $J_{K,n}(q)$ colored by the $n$-dimensional irreducible
$\mathfrak{sl}_2(\BC)$ representation. For simplicity, we abbreviate
$R\big(S^3\setminus K\big)$ defined above by~$R(K)$.

\begin{Question}\label{que.JK}\quad
\begin{enumerate}\itemsep=0pt
\item[$(a)$]
Does the vector space $R(K)$ have a ring structure?
\item[$(b)$]
If so, is the series $\sum_{n=1}^\infty J_{K,n}(q) Q^n$ the coefficient
of $1$ in the $R(K)$-valued small $J$-function $J_K(Q,q,0)$ of a knot $K$?
\item[$(c)$]
If so, is there a $t$-deformation $J_K(Q,q,t)$?
\end{enumerate}
\end{Question}

\subsection*{Acknowledgements}

The authors wish to thank the Max-Planck-Institute for Mathematics and the
Bethe Center for Theoretical Physics in Bonn for inviting them to their
workshop on \emph{Number Theoretic Me\-thods in Quantum Physics} in July
2019, where the first ideas were conceived. We also wish to thank
Gaetan Borot, Alexander Givental, Todor Milanov and Di Yang for useful
conversations. E.S.~wishes to thank the University of Melbourne for
having him as a guest during 2020 and Southern University of Science
and Technology for hospitality in 2021.

\pdfbookmark[1]{References}{ref}
\LastPageEnding


\begin{thebibliography}{99}
\footnotesize\itemsep=0pt

\bibitem{Aganagic:2012jb}
Aganagic M., Vafa C., Large $N$ duality, mirror symmetry, and a $Q$-deformed
 $A$-polynomial for knots, \href{https://arxiv.org/abs/1204.4709}{arXiv:1204.4709}.

\bibitem{AvEvSZ}
Almkvist G., van Enckevort C., van Straten D., Zudilin W., Tables of
 {C}alabi--{Y}au equations, \href{https://arxiv.org/abs/math.AG/0507430}{arXiv:math.AG/0507430}.

\bibitem{AK:TQFT}
Andersen J.E., Kashaev R., A {TQFT} from quantum {T}eichm\"uller theory,
 \href{https://doi.org/10.1007/s00220-014-2073-2}{\textit{Comm. Math. Phys.}} \textbf{330} (2014), 887--934, \href{https://arxiv.org/abs/1109.6295}{arXiv:1109.6295}.

\bibitem{AK:ICM}
Andersen J.E., Kashaev R., The {T}eichm\"uller {TQFT}, in Proceedings of the
 {I}nternational {C}ongress of {M}athe\-maticians~-- {R}io de {J}aneiro 2018,
 {V}ol.~{III}, {I}nvited Lectures, \href{https://doi.org/10.1142/9789813272880_0149}{World Sci. Publ.}, Hackensack, NJ, 2018,
 2541--2565, \href{https://arxiv.org/abs/1811.06853}{arXiv:1811.06853}.

\bibitem{Anderson:finiteness}
Anderson D., Chen L., Tseng H.H., On the finiteness of quantum {K}-theory of a
 homogeneous space, \href{https://doi.org/10.1093/imrn/rnaa108}{\textit{Int. Math. Res. Not.}} \textbf{2022} (2022),
 1313--1349, \href{https://arxiv.org/abs/1804.04579}{arXiv:1804.04579}.

\bibitem{holo-blocks}
Beem C., Dimofte T., Pasquetti S., Holomorphic blocks in three dimensions,
 \href{https://doi.org/10.1007/JHEP12(2014)177}{\textit{J.~High Energy Phys.}} \textbf{2014} (2014), no.~12, 177, 119~pages,
 \href{https://arxiv.org/abs/1211.1986}{arXiv:1211.1986}.

\bibitem{Candelas:pair}
Candelas P., de~la Ossa X.C., Green P.S., Parkes L., A pair of {C}alabi--{Y}au
 manifolds as an exactly soluble superconformal theory, \href{https://doi.org/10.1016/0550-3213(91)90292-6}{\textit{Nuclear
 Phys.~B}} \textbf{359} (1991), 21--74.

\bibitem{CCGLS}
Cooper D., Culler M., Gillet H., Long D.D., Shalen P.B., Plane curves
 associated to character varieties of {$3$}-manifolds, \href{https://doi.org/10.1007/BF01231526}{\textit{Invent. Math.}}
 \textbf{118} (1994), 47--84.

\bibitem{Cox-Katz}
Cox D.A., Katz S., Mirror symmetry and algebraic geometry, \textit{Mathematical
 Surveys and Monographs}, Vol.~68, \href{https://doi.org/10.1090/surv/068}{Amer. Math. Soc.}, Providence, RI, 1999.

\bibitem{Desole-Kac}
De~Sole A., Kac V.G., On integral representations of {$q$}-gamma and {$q$}-beta
 functions, \textit{Atti Accad. Naz. Lincei Cl. Sci. Fis. Mat. Natur. Rend.
 Lincei~(9) Mat. Appl.} \textbf{16} (2005), 11--29, \href{https://arxiv.org/abs/math.QA/0302032}{arXiv:math.QA/0302032}.

\bibitem{Dimofte:complexCS}
Dimofte T., Complex {C}hern--{S}imons theory at level {$k$} via the 3d-3d
 correspondence, \href{https://doi.org/10.1007/s00220-015-2401-1}{\textit{Comm. Math. Phys.}} \textbf{339} (2015), 619--662,
 \href{https://arxiv.org/abs/1409.0857}{arXiv:1409.0857}.

\bibitem{Dimofte:perturbative}
Dimofte T., Perturbative and nonperturbative aspects of complex
 {C}hern--{S}imons theory, \href{https://doi.org/10.1088/1751-8121/aa6a5b}{\textit{J.~Phys.~A: Math. Theor.}} \textbf{50}
 (2017), 443009, 25~pages, \href{https://arxiv.org/abs/1608.02961}{arXiv:1608.02961}.

\bibitem{DGG2}
Dimofte T., Gaiotto D., Gukov S., 3-manifolds and 3d indices, \href{https://doi.org/10.4310/ATMP.2013.v17.n5.a3}{\textit{Adv.
 Theor. Math. Phys.}} \textbf{17} (2013), 975--1076, \href{https://arxiv.org/abs/1112.5179}{arXiv:1112.5179}.

\bibitem{DGG1}
Dimofte T., Gaiotto D., Gukov S., Gauge theories labelled by three-manifolds,
 \href{https://doi.org/10.1007/s00220-013-1863-2}{\textit{Comm. Math. Phys.}} \textbf{325} (2014), 367--419, \href{https://arxiv.org/abs/1108.4389}{arXiv:1108.4389}.

\bibitem{Dimofte:vortex}
Dimofte T., Gukov S., Hollands L., Vortex counting and {L}agrangian
 3-manifolds, \href{https://doi.org/10.1007/s11005-011-0531-8}{\textit{Lett. Math. Phys.}} \textbf{98} (2011), 225--287,
 \href{https://arxiv.org/abs/1006.0977}{arXiv:1006.0977}.

\bibitem{Fa}
Faddeev L.D., Discrete {H}eisenberg--{W}eyl group and modular group,
 \href{https://doi.org/10.1007/BF01872779}{\textit{Lett. Math. Phys.}} \textbf{34} (1995), 249--254,
 \href{https://arxiv.org/abs/hep-th/9504111}{arXiv:hep-th/9504111}.

\bibitem{Ga:AJ}
Garoufalidis S., On the characteristic and deformation varieties of a knot, in
 Proceedings of the {C}asson {F}est, \textit{Geom. Topol. Monogr.}, Vol.~7,
 \href{https://doi.org/10.2140/gtm.2004.7.291}{Geom. Topol. Publ.}, Coventry, 2004, 291--309, \href{https://arxiv.org/abs/math.GT/0306230}{arXiv:math.GT/0306230}.

\bibitem{GK:qseries}
Garoufalidis S., Kashaev R., From state integrals to {$q$}-series,
 \href{https://doi.org/10.4310/MRL.2017.v24.n3.a8}{\textit{Math. Res. Lett.}} \textbf{24} (2017), 781--801, \href{https://arxiv.org/abs/1304.2705}{arXiv:1304.2705}.

\bibitem{GL:qholo}
Garoufalidis S., L\^e T.T.Q., The colored {J}ones function is {$q$}-holonomic,
 \href{https://doi.org/10.2140/gt.2005.9.1253}{\textit{Geom. Topol.}} \textbf{9} (2005), 1253--1293,
 \href{https://arxiv.org/abs/math.GT/0309214}{arXiv:math.GT/0309214}.

\bibitem{GZ:qmodularity}
Garoufalidis S., Zagier D., Knots, perturbative series and quantum modularity,
 \href{https://arxiv.org/abs/2111.06645}{arXiv:2111.06645}.

\bibitem{GZ:qseries}
Garoufalidis S., Zagier D., Knots and their related $q$-series, in preparation.

\bibitem{Givental:WDVV}
Givental A., On the {WDVV} equation in quantum {$K$}-theory, \href{https://doi.org/10.1307/mmj/1030132720}{\textit{Michigan Math.~J.}} \textbf{48} (2000), 295--304, \href{https://arxiv.org/abs/math.AG/0003158}{arXiv:math.AG/0003158}.

\bibitem{Givental:equivariantV}
Givental A., Permutation-equivariant quantum {K}-theory~{V}. {T}oric
 $q$-hypergeometric functions, \href{https://arxiv.org/abs/1509.03903}{arXiv:1509.03903}.

\bibitem{Givental:equivariantVIII}
Givental A., Permutation-equivariant quantum {K}-theory {VIII}. {E}xplicit
 reconstruction, \href{https://arxiv.org/abs/1510.06116}{arXiv:1510.06116}.

\bibitem{Givental-Reconstruction}
Givental A., Explicit reconstruction in quantum cohomology and {K}-theory,
 \href{https://doi.org/10.5802/afst.1500}{\textit{Ann. Fac. Sci. Toulouse Math.}} \textbf{25} (2016), 419--432,
 \href{https://arxiv.org/abs/1506.06431}{arXiv:1506.06431}.

\bibitem{Givental-Lee}
Givental A., Lee Y.-P., Quantum {$K$}-theory on flag manifolds,
 finite-difference {T}oda lattices and quantum groups, \href{https://doi.org/10.1007/s00222-002-0250-y}{\textit{Invent. Math.}}
 \textbf{151} (2003), 193--219, \href{https://arxiv.org/abs/math.AG/0108105}{arXiv:math.AG/0108105}.

\bibitem{Givental-Tonita}
Givental A., Tonita V., The {H}irzebruch--{R}iemann--{R}och theorem in true
 genus-0 quantum {K}-theory, in Symplectic, {P}oisson, and noncommutative
 geometry, \textit{Math. Sci. Res. Inst. Publ.}, Vol.~62, Cambridge University
 Press, New York, 2014, 43--91, \href{https://arxiv.org/abs/1106.3136}{arXiv:1106.3136}.

\bibitem{Iritani:rec}
Iritani H., Milanov T., Tonita V., Reconstruction and convergence in quantum
 {$K$}-theory via difference equations, \href{https://doi.org/10.1093/imrn/rnu026}{\textit{Int. Math. Res. Not.}}
 \textbf{2015} (2015), 2887--2937, \href{https://arxiv.org/abs/1309.3750}{arXiv:1309.3750}.

\bibitem{Jockers:qK}
Jockers H., Mayr P., Quantum {K}-theory of {C}alabi--{Y}au manifolds,
 \href{https://doi.org/10.1007/jhep11(2019)011}{\textit{J.~High Energy Phys.}} \textbf{2019} (2019), no.~011, 20~pages,
 \href{https://arxiv.org/abs/1905.03548}{arXiv:1905.03548}.

\bibitem{Jockers:3d}
Jockers H., Mayr P., A 3d gauge theory/quantum {K}-theory correspondence,
 \href{https://doi.org/10.4310/ATMP.2020.v24.n2.a4}{\textit{Adv. Theor. Math. Phys.}} \textbf{24} (2020), 327--457,
 \href{https://arxiv.org/abs/1808.02040}{arXiv:1808.02040}.

\bibitem{KLV}
Kashaev R., Luo F., Vartanov G., A {TQFT} of {T}uraev--{V}iro type on shaped
 triangulations, \href{https://doi.org/10.1007/s00023-015-0427-8}{\textit{Ann. Henri Poincar\'e}} \textbf{17} (2016),
 1109--1143, \href{https://arxiv.org/abs/1210.8393}{arXiv:1210.8393}.

\bibitem{K95}
Kashaev R.M., A link invariant from quantum dilogarithm, \href{https://doi.org/10.1142/S0217732395001526}{\textit{Modern Phys.
 Lett.~A}} \textbf{10} (1995), 1409--1418, \href{https://arxiv.org/abs/q-alg/9504020}{arXiv:q-alg/9504020}.

\bibitem{Lee:foundations}
Lee Y.-P., Quantum {$K$}-theory. {I}.~{F}oundations, \href{https://doi.org/10.1215/S0012-7094-04-12131-1}{\textit{Duke Math.~J.}}
 \textbf{121} (2004), 389--424, \href{https://arxiv.org/abs/math.AG/0105014}{arXiv:math.AG/0105014}.

\bibitem{Macdonald}
Macdonald I.G., Symmetric functions and {H}all polynomials, 2nd~ed., \textit{Oxford
 Mathematical Monographs}, The Clarendon Press, Oxford University Press, New
 York, 1995.

\bibitem{PWZ}
Petkov\v{s}ek M., Wilf H.S., Zeilberger D., $A=B$, A.K.~Peters Ltd., Wellesley,
 MA, 1996.

\bibitem{RT}
Reshetikhin N., Turaev V.G., Invariants of {$3$}-manifolds via link polynomials
 and quantum groups, \href{https://doi.org/10.1007/BF01239527}{\textit{Invent. Math.}} \textbf{103} (1991), 547--597.

\bibitem{Taipale}
Taipale K., K-theoretic {J}-functions of type~{A} flag varieties, \href{https://doi.org/10.1093/imrn/rns156}{\textit{Int.
 Math. Res. Not.}} \textbf{2013} (2013), 3647--3677, \href{https://arxiv.org/abs/1110.3117}{arXiv:1110.3117}.

\bibitem{Tonita:quintic}
Tonita V., Twisted {K}-theoretic {G}romov--{W}itten invariants, \href{https://doi.org/10.1007/s00208-018-1659-y}{\textit{Math.
 Ann.}} \textbf{372} (2018), 489--526, \href{https://arxiv.org/abs/1508.05976}{arXiv:1508.05976}.

\bibitem{Tu:book}
Turaev V.G., Quantum invariants of knots and 3-manifolds, \textit{De Gruyter
 Studies in Mathematics}, Vol.~18, \href{https://doi.org/10.1515/9783110435221}{Walter de Gruyter \& Co.}, Berlin, 1994.

\bibitem{Wen}
Wen Y., Difference equation for quintic 3-fold, \href{https://arxiv.org/abs/2011.07527}{arXiv:2011.07527}.

\bibitem{WZ}
Wilf H.S., Zeilberger D., An algorithmic proof theory for hypergeometric
 (ordinary and ``{$q$}'') multisum/integral identities, \href{https://doi.org/10.1007/BF02100618}{\textit{Invent. Math.}}
 \textbf{108} (1992), 575--633.

\bibitem{Wi:CS}
Witten E., Quantum field theory and the {J}ones polynomial, \href{https://doi.org/10.1007/BF01217730}{\textit{Comm. Math.
 Phys.}} \textbf{121} (1989), 351--399.

\bibitem{Zeilberger}
Zeilberger D., A holonomic systems approach to special functions identities,
 \href{https://doi.org/10.1016/0377-0427(90)90042-X}{\textit{J.~Comput. Appl. Math.}} \textbf{32} (1990), 321--368.

\end{thebibliography}
\end{document}